\documentclass[11pt,oneside]{article}
\usepackage[utf8]{inputenc}
\usepackage[english]{babel}
\usepackage{amsmath}
\usepackage{amsfonts}
\usepackage{amssymb}
\usepackage{amsthm}
\usepackage{tikz-cd} 
\usepackage{hyperref}
\usepackage{pgfplots}
\usetikzlibrary{calc}

\usepackage[left=3cm,right=3cm,top=3cm,bottom=3cm]{geometry}
\usepackage[title]{appendix}
\usepackage[]{algorithm2e}

\usepackage[normalem]{ulem}
\useunder{\uline}{\ul}{}
\theoremstyle{plain}
\newtheorem{teo}{}[section]
\newtheorem{prop}[teo]{Proposition}
\newtheorem{cor}[teo]{Corollary}

\newtheorem{lem}[teo]{Lemma}
\newtheorem{thm}[teo]{Theorem}
\theoremstyle{definition}
\newtheorem{ex}[teo]{Example}
\newtheorem{rem}[teo]{Remark}
\newtheorem{df}[teo]{Definition}
\usepackage{graphicx} 

\newcommand\blfootnote[1]{%
  \begingroup
  \renewcommand\thefootnote{}\footnote{#1}%
  \addtocounter{footnote}{-1}%
  \endgroup
}

\title{On continuous homomorphisms from Alexandroff paratopological groups into topological groups}
\author{Tayomara Borsich and Pedro J. Chocano}
\date{}

\begin{document}

\maketitle

\begin{abstract}
Alexandroff paratopological groups provide a natural setting in which order-theoretic and algebraic properties interact. In this short note we prove obstruction results when considering continuous homomorphisms from Alexandroff paratopological groups into topological groups. Particularly and, as a consequence of these results, we provide an alternative proof of the fact that the only connected Alexandroff topological group is the one that carries the trivial topology.
\end{abstract}

\blfootnote{This research is partially supported by Grant  PID2024-156663NB-I00 from Ministerio de Ciencia, Innovación y Universidades (Spain) and  2025/SOLCON-159637 from Rey Juan Carlos University}
\section{Introduction}
Alexandroff spaces, and in particular finite topological spaces, have proved to be very useful tools for addressing problems of diverse nature in mathematics. For instance, shape theory can be described by means of inverse sequences of finite topological spaces~\cite{chocano2025shape}. Moreover, several realization problems for algebraic structures via topological spaces have been answered positively (see, for example, \cite{chocano2025realizingisometries,chocano2025realizing,chocano2020homomorphisms} and the references therein).

Motivated by these developments, the study of algebraic structures endowed with Alexandroff topologies has emerged as a natural direction of research. In~\cite{borsich2026paratopological}, the authors investigated Alexandroff topologies on groups and proved that no non-trivial topological group admits such a structure unless the topology is either trivial or discrete. This naturally leads to the broader context of \emph{paratopological groups}, where the multiplication is required to be continuous while the inversion need not be. In this setting, non-trivial Alexandroff examples do exist, and they provide a flexible framework for exploring interactions between order-theoretic and algebraic properties.

The purpose of this short note is to contribute to the structural understanding of Alexandroff paratopological groups by establishing general rigidity results. Our starting point is the following situation: let $X$ be an Alexandroff paratopological group, let $G$ be a topological group, and consider a continuous homomorphism

\[
f : X \longrightarrow G.
\]

Under these hypotheses, one may ask what restrictions the Alexandroff topology on $X$ imposes on the topology of $G$ and on the behaviour of $f$. In particular, additional assumptions on $X$—such as connectedness or separation axioms—turn out to force strong constraints on the image of $f$. One of our main results shows that no connected Alexandroff $T_0$ paratopological group can be continuously isomorphic to a non-trivial topological group. This phenomenon illustrates a sharp contrast between Alexandroff paratopological groups and classical topological groups, highlighting how the Alexandroff condition severely limits the existence of non-degenerate continuous homomorphisms.

A related rigidity phenomenon appears in the classical context of continuous maps from finite Alexandroff spaces to compact Hausdorff spaces: as shown in \cite[Proof of Proposition~2.6]{chocano2025categoria}, a finite Alexandroff $T_0$-space admits no non-constant continuous map into a compact Hausdorff space. Our results can thus be viewed as an algebraic counterpart of this topological obstruction.

The paper is organized as follows. In Section~\ref{sec_preliminaries}, we introduce the basic terminology and results needed to make the paper as self contained as possible. In Section~\ref{sec_main}, we develop the main results of this paper.

\section{Preliminaries}\label{sec_preliminaries}

For basic information on Alexandroff spaces we refer the reader to \cite{may1966finite,barmak2011algebraic}. The results recalled below can be found in greater detail in those references.

\begin{df}
    An \emph{Alexandroff space} is a topological space in which arbitrary intersections of open sets are open.
\end{df}

Let $X$ be an Alexandroff space. For each $x\in X$, denote by $U_x$ the intersection of all open sets containing $x$; this is the smallest open neighbourhood of $x$. Similarly, let $F_x$ be the intersection of all closed sets containing $x$, which is the smallest closed neighbourhood of $x$. Given two points $x,y\in X$, declare $x\leq y$ if and only if $U_x\subseteq U_y$. It is straightforward to verify that this relation is a preorder. Moreover, if $X$ is a $T_0$-space, then the above relation is a partial order on $X$.

Conversely, if $(X,\leq)$ is a preordered set, then the collection of lower sets forms a basis for an Alexandroff topology on $X$. Moreover, if $\leq$ is a partial order, then the topology is $T_0$.  These two constructions are mutually inverse. Therefore, there is a one-to-one correspondence between preorders and Alexandroff topologies on a set $X$. Particularly,

\begin{thm}
    For a set $X$, the Alexandroff space topologies on $X$ are in
bijective correspondence with the preorders on X. The topology corresponding
to $\leq$ is $T_0$ if and only if the relation $\leq$ is a partial order.
\end{thm}

If $f\colon X\to Y$ is a map between Alexandroff spaces, then $f$ is continuous if and only if it is order-preserving. Hence:

\begin{thm}
    The category of (partially ordered) preordered sets and the category of Alexandroff ($T_0$) topological spaces are isomorphic.
\end{thm}

From now on, we do not distinguish between an Alexandroff space and its associated preorder. In particular, we freely use order-theoretic notation. For instance,

\[
    U_x=\{y\in X : y\leq x\}
    \qquad\text{and}\qquad
    F_x=\{y\in X : y\geq x\}.
\]

An important property for these spaces that will be used in the subsequent section is the following:

\begin{prop}\label{prop_caracterizacion_conexion}
    Let $X$ be an Alexandroff space. Then $X$ is connected if and only if $X$ is path-connected. Moreover, if $X$ is connected and $x,y\in X$, then there is a finite sequence of points $z_i$, $1 \leq i \leq n$, such
    that $z_1 = x,~ z_n = y$ and either $z_i \leq  z_{i+1}$ or $z_{i+1} \leq  z_i$ for $i < n$.
\end{prop}

Moreover, we recall a stronger notion of connectivity that plays a central role in this manuscript.

\begin{df}
    A topological space $X$ is said to be \emph{hyperconnected} (resp., \emph{ultraconnected}) if the intersection of every pair of non-empty open sets (resp., closed sets) is non-empty. Equivalently, $X$ has no two disjoint non-empty open subsets (resp., closed subsets).
\end{df}

In general, we cannot expect a connected Alexandroff topological space to be hyperconnected. For instance, consider  $X=\{x_1,x_2,x_3\}$ with the order relations \(x_1 < x_2 > x_3\). This is an Alexandroff \(T_0\) space. We have $U_{x_1}=\{x_1\}$, $U_{x_3}=\{x_3\}$, and therefore \(U_{x_1}\cap U_{x_3}=\emptyset\). Hence \(X\) is not hyperconnected. However, it is easy to verify that \(X\) is ultraconnected. If we now consider the four‑point Alexandroff space $X=\{x_1,x_2,x_3,x_4\}$ with the relations \(x_1 < x_2 > x_3< x_4\), we obtain a connected space that is neither hyperconnected nor ultraconnected.


For an introduction to the theory of topological groups and paratopological groups that are Hausdorff we refer to \cite{TArhangelskiiTkachenko2008}.

\begin{df}
    Let $G$ be a group endowed with a topology. We say that $G$ is a \emph{topological group} if:
    \begin{enumerate}
        \item the multiplication map $m\colon G\times G\to G$ is continuous, and
        \item the inversion map $\mathrm{Inv}\colon G\to G$ is continuous.
    \end{enumerate}
    We say that $G$ is a \emph{paratopological group} if only the first condition holds.
\end{df}

Note that in the definition of topological group, as well as in the definition of paratopological group, no separation axioms are assumed. Throughout this paper, unless a specific separation property is explicitly stated, we impose none.

\textbf{Notation.} Given a (paratopological) topological group $G$, we denote by $e_G$ its identity element. When no confusion can arise we simple denote the identity element by $e$. Moreover, $\overline{\{e_G\}}^G$ denotes the closure of $e_G$ in $G$. Given a subset $C$ of $G$, we denote  $\langle C \rangle$ to the subgroup generated by $C$.

\section{Main results}\label{sec_main}

We begin by generalizing a property proved in \cite{borsich2026paratopological} to the setting of Alexandroff spaces that are not necessarily \(T_0\). The proof follows the same structure as the one given in \cite[Theorem~4.17]{borsich2026paratopological}, but it is more elementary. Nevertheless, for the sake of simplicity and in order to keep the paper as self-contained as possible, we have decided to include it.

\begin{thm}\label{thm_hyperconnected}
    Let $X$ be a connected Alexandroff paratopological group. Then $X$ is hyperconnected and ultraconnected.
\end{thm}
\begin{proof}
    Let $x,y \in X$. By the connectedness of $X$ and Proposition \ref{prop_caracterizacion_conexion}, we may find a minimal sequence of points $z_1,\ldots,z_n \in X$ such that $z_1 = x$, $z_n = y$, and for each $i = 0,\ldots,n-1$ we have either $z_i < z_{i+1}$ or $z_i > z_{i+1}$.  
    The minimality of the sequence guarantees that for each $i$,
\[
z_i < z_{i+1} > z_{i+2} \quad \text{or} \quad z_i > z_{i+1} < z_{i+2},
\]
since otherwise we could have $z_i < z_{i+1} < z_{i+2}$, which means that we may delete the middle point, contradicting minimality.

Without loss of generality, assume that
\[
z = z_1 < z_2 > z_3 < z_4 > \cdots > z_n = y.
\]
We prove that there exists \(w_2 \in X\) such that \(z_1 > w_2 < z_3\). Multiplying the relation \(z_1 < z_2 > z_3\) by \(z_2^{-1}\), we obtain
\[
z_1 z_2^{-1} < e > z_3 z_2^{-1}.
\]
Since \(z_1 z_2^{-1} < e\), we have \(e < z_2 z_1^{-1}\). Similarly, from
\(z_3 z_2^{-1} < e\), we obtain \(e < z_2 z_3^{-1}\). Now multiply \(z_2 z_1^{-1} > e < z_2 z_3^{-1}\) by \(z_2^{-1}\) to get
\[
z_1^{-1} > z_2^{-1} < z_3^{-1}.
\]
Finally, first multiply this relation by \(z_1\) and then by \(z_3\), obtaining
\[
z_1 > z_1 z_2^{-1} z_3 < z_3.
\]
Hence, we may define $w_2 := z_1 z_2^{-1} z_3$.

From the argument above, we may replace \(z_2\) by \(w_2\) in the sequence
connecting \(z_1\) with \(z_n\). However, this implies that we can omit \(w_2\),
contradicting the minimality of the sequence unless \(z_3 = y\). In that case,
we obtain \(U_x \cap U_y \neq \emptyset\), as desired. 

The same arguments can be applied to prove that $X$ is ultraconnected.

\end{proof}

We now prove a technical lemma that will play a key role in subsequent results.
\begin{lem}\label{lem_revisor}
Let $X$ be a paratopogical group. Suppose that the subgroup generated by $\overline{\{e_X\}}^X$ is $X$. Then for any continuous homomorphism $f:X\rightarrow G$ into a topological group $G$, the image $f(X)$ carries the trivial topology.
\end{lem}
\begin{proof}
    It is clear that $f(e_X)=e_G$ because $f$ is a homomorphism. Since $f$ is continuous, one has that $$f(\overline{\{e_X\}}^X)\subseteq \overline{\{e_G\}}^G.$$ On the other hand, $\overline{\{e_G\}}^G$ is a subgroup of $G$ because $G$ is a topological group. This implies that $$f(X)=\langle f(\overline{\{e_X\}}^X)\rangle\subseteq \langle\overline{\{e_G\}}^G\rangle= \overline{\{e_G\}}^G. $$

    Thus, $f(X)$ carries the trivial topology. Let us prove the last assertion. For every $z\in \overline{\{e_G \}}^G$, one can consider the map $m(z,\cdot):G\rightarrow G$ given by $m(z,y)=zy$, which is a homeomorphism. Therefore, $$\overline{\{z \}}^G=\overline{\{m(z,e_G)\}}^G=m(z,\overline{\{e_G \}}^G)=z \overline{\{e_G \}}^G=\overline{\{e_G \}}^G.$$
    This shows that $\overline{\{e_G \}}^G$ carries the trivial topology and consequently $f(X)$ as a subspace also inherits the trivial topology.
     
\end{proof}

\begin{lem}\label{lema_U_e_F_e_son_generadores}
    Let $X$ be a connected Alexandroff paratopological group. Then $\langle F_e \rangle=\langle U_e \rangle=X$.
\end{lem}
\begin{proof}
    We prove that  $\langle F_e \rangle$ is the whole group.  It is evident that $F_{e}\cup U_{e}\subseteq \langle F_{e} \rangle$. Suppose that $z\notin F_{e}\cup U_e$. By Theorem \ref{thm_hyperconnected}, $U_{e}\cap U_z\neq \emptyset$, which implies that there exists $y\in U_{e}\cap U_z$. Particularly, this means that $y\leq e$ and $y\leq z$. By the continuity of the multiplication map: $e\leq y^{-1}z$, which means that $y^{-1}z\in F_{e}$. Thus, $yy^{-1}z=z\in  \langle F_{e}\rangle$. We only need to repeat the same arguments to show that $\langle U_e \rangle=X$.
\end{proof}

We now have all the ingredients to prove one of our main results:

\begin{thm}\label{thm_conexo_trivial}
    Let $X$ be a connected Alexandroff paratopological group. If $f:X\rightarrow G$ is a continuous homomorphism into a topological group $G$, then $f(X)$ carries the trivial topology.
\end{thm}
\begin{proof}
    It is an immediate consequence of Lemma \ref{lem_revisor}, Lemma \ref{lema_U_e_F_e_son_generadores} and the fact that the closure of the identity element $e\in X$ corresponds to $F_e$. 
\end{proof}

This result gives as an immediate consequence the following corollaries:

\begin{cor}
    Let $X$ be a connected Alexandroff paratopological space. If $f:X\rightarrow G$ is a continuous epimorphism (monomorphism) into a topological group $G$, then $G$ carries the trivial topology. Particularly, there is no continuous isomorphism from $X$ onto $G$ unless $G$ carries the trivial topology.
\end{cor}

\begin{cor}\label{cor_homo_trivial}
    Let $X$ be a connected Alexandroff paratopological space. If $f:X\rightarrow G$ is a continuous homomorphism into a $T_1$ topological group. Then $f$ is the trivial homomorphism.
\end{cor}
\begin{proof}
    Since $T_1$ is equivalent to have closed points, we have that $\overline{\{ e_G\}}^G=e_G$. From the proof of Theorem \ref{thm_conexo_trivial}, we conclude that $f(X)=e_G$.
\end{proof}

\begin{cor}\label{cor_no_alexandroff_conexo_grupo}
    Let $X$ be a group. Then there is only one connected Alexandroff topology that turns $X$ into a topological group and is the trivial topology.
\end{cor}
\begin{proof}
    Let us suppose that $X$ is a connected Alexandroff topological group. Consider the identity map $\textnormal{id}:X\rightarrow X$. By Theorem \ref{thm_conexo_trivial}, we conclude that $X$ carries the trivial topology.
\end{proof}

\begin{cor}\label{cor_topologia_trivial_componentes}
    Let $X$ be a group. Then the only Alexandroff topology that turns $X$ into a topological group satisfies that is the trivial topology on each connected component.
\end{cor}
\begin{proof}
Assume $X$ is a topological group. Let $X_1$ denote the connected component that contains the identity element. By Corollary \ref{cor_no_alexandroff_conexo_grupo}, $X_1$ carries the trivial topology. Moreover, let $X_j$ be another connected component of $X$. Consider $x\in X_j$ and the map $m(x,\cdot):X\rightarrow X$ given by $m(x,y)=xy$, i.e., the multiplication map. This map is a homeomorphism, which implies that $X_j$ also carries the trivial topology.
\end{proof}

Note that Corollary~\ref{cor_no_alexandroff_conexo_grupo} and Corollary~\ref{cor_topologia_trivial_componentes} complement and extend the triviality results obtained in \cite{borsich2026paratopological} for Alexandroff \(T_0\) topological groups. In particular, in \cite[Theorem~3.6]{borsich2026paratopological} the authors proved that the only connected Alexandroff \(T_0\) topological group is the trivial one-point topological space.

Moreover, Corollary~\ref{cor_topologia_trivial_componentes} can be used to classify (topologically) Alexandroff topological groups just by their number of connected components. Let us recall that the space of connected component of a topological space $X$ is denoted by $\pi_0(X)$, we endow the discrete topology to that space.

\begin{cor}
    Let $X$ be an Alexandroff topological group. Then \(X\) has the same homotopy type as $\pi_0(X)$.
\end{cor}
\begin{proof}
    It is an immediate consequence of  Corollary~\ref{cor_topologia_trivial_componentes} and the fact that a topological space with the trivial topology has the same homotopy type as a one-point space.
\end{proof}

We now analyze the situation when the connectedness hypothesis is removed. We start by proving a general result within the theory of Alexandroff spaces.

\begin{lem}\label{lem_aplicacion_constante}
    Let $X$ be an Alexandroff space and let $Y$ be a $T_1$ space. If $f:X\rightarrow Y$ is a continuous map, then $f$ is the constant map on each connected component. 
\end{lem}
\begin{proof}
    Suppose $X_j$ are the connected components of $X$ for some index set $J$. Since $Y$ is a $T_1$ space, the points are closed sets. Hence, for every $x\in X_j$, we have that $f(F_x)=f(x)$. By the connectedness of $X_j$ and Proposition \ref{prop_caracterizacion_conexion}, for any two points $x,y\in X_j$, there exists a finite sequence of points $z_0,...,z_n$ such that $z_0=x$, $z_n=y$ and $z_i\leq z_{i+1}$ or $z_{i+1}\leq z_i$ for every $i=0,...,n-1$. From this sequence and  the fact that $f(F_x)=f(x)$ we may conclude the desired result.
\end{proof}

Although Lemma \ref{lem_aplicacion_constante} and Corollary \ref{cor_homo_trivial} could lead to think that it is possible to obtain a generalization of Corollary \ref{cor_homo_trivial} for the non-connected case as we demonstrate with the following example this is not possible.

\begin{ex}\label{ex_no_trivial}
    Let us consider the standard group of integer numbers $\mathbb{Z}$ and the cyclic group of two elements $\mathbb{Z}_2$. Consider the group $X=\mathbb{Z}\times \mathbb{Z}_2$ provided by the direct product of $\mathbb{Z}$ with $ \mathbb{Z}_2$. Define the relation  $\leq $ by $(a,b)\leq (c,d)$ if and only if $b=d$ and  $a\leq_u c$, where $\leq_u$ stands for the usual order of the integer numbers. It is routine to verify that $X$ is a non-connected Alexandroff $T_0$ paratopological group. The connected components of $X$ are $\mathbb{Z}\times \{0 \}$ and $\mathbb{Z}\times \{1\}$. On the other hand, consider the standard topological group of the circle $G=S^1$ as a subspace of the complex plane. Define $f:X\rightarrow G$ by $f(a,0)=1$ and $f(a,1)=-1$ for every $a\in \mathbb{Z}$. It is easy to verify that $f$ is a continuous homomorphism of groups.
\end{ex}
\begin{rem}
 Note that Example~\ref{ex_no_trivial} also illustrates the importance of the separation axioms. In this example, the space $X$ is $T_0$, and the topology on each connected component is neither trivial nor discrete. This does not contradict Corollary~\ref{cor_topologia_trivial_componentes}, since $X$ is an Alexandroff paratopological group rather than a topological group.
\end{rem}

Let us prove that in the case of a non-connected Alexandroff paratopological $X$, the image, $f(X)$, of a continuous homomorphism $f:X\rightarrow G$ into a topological group need not to carry the discrete topology.

\begin{ex} Let us consider the standard groups of integer numbers $\mathbb{Z}$ and the rational numbers $\mathbb{Q}$. Consider the direct product of these two groups $X=\mathbb{Z}\times \mathbb{Q}$. Declare that $(a,b)\leq (c,d)$ if and only if $a\leq_u c$ and $b=d$, where $\leq_u$ denotes the usual order of the integer numbers. It is easy to verify that $X$ is a non-connected Alexandroff paratopological group and the connected components of $X$ are provided by $\mathbb{Z}\times \{ b\}$ where $b\in \mathbb{Q}$. Consider the topological group of the real numbers with the sum. Define $f:X\rightarrow \mathbb{R}$ by $f(a,b)=b$. Hence, $f$ is a continuous homomorphism and $f(X)$ does not carry the discrete topology. 
    
\end{ex}

Finally, we exhibit a continuous homomorphism from a paratopological group to an Alexandroff paratopological group.
\begin{ex} Let us consider the  Sorgenfrey line $G=\mathbb{R}$. This is a non-Alexandroff paratopological group. On the other hand, consider the Alexandroff paratopological group of the real numbers with the usual order and standard sum $X=\mathbb{R}$. The identity map $\textnormal{id}:G\rightarrow X$ is a continuous homomorphism.
\end{ex}

To conclude, we state a result for the case of a finite number of connected components that can be deduced easily.
\begin{prop}
    Let $X$ be an Alexandroff paratopological group with a finite number of connected components and let $G$ be a $T_1$ topological group. If $f:X\rightarrow G$ is a continuous homomorphism, then $f(X)$ is a discrete topological group. 
\end{prop}

\subsubsection*{Acknowledgment} We are very grateful to the anonymous reviewer of a previous version of this manuscript, who pointed out to us a version of Lemma \ref{lem_revisor}. This result plays a key role in the present paper and in several subsequent developments.




\bibliography{bibliografia}
\bibliographystyle{plain}

\newcommand{\Addresses}{{
  \bigskip
  \footnotesize

    \textsc{T. Borsich, Departamento de Matemática Aplicada, Ciencia e Ingeniería de los Materiales y Tecnología Electrónica, ESCET Universidad Rey Juan Carlos, 28933 Móstoles (Madrid), Spain}\par\nopagebreak
 \textit{E-mail address}: \texttt{tayomara.gonzalez@urjc.es}

  \textsc{ P.J. Chocano, Departamento de Matemática Aplicada, Ciencia e Ingeniería de los Materiales y Tecnología Electrónica, ESCET Universidad Rey Juan Carlos, 28933 Móstoles (Madrid), Spain}\par\nopagebreak
 \textit{E-mail address}: \texttt{pedro.chocano@urjc.es}

}}

\Addresses

\end{document}